\documentclass[12pt]{article}
\usepackage{cite}
\usepackage{amsmath,amsthm,amssymb,amsfonts,amscd}
\usepackage{enumerate}
\usepackage{graphicx}
\usepackage{graphics}
\usepackage{subfigure}
\usepackage{mathrsfs}
\usepackage{a4,latexsym,epsfig}
\usepackage{color,colordvi}
\usepackage{tikz}
\usepackage{fullpage}
\newtheorem{theorem}{Theorem}[section]
\newtheorem{proposition}[theorem]{Proposition}%[section]
\newtheorem{definition}{Definition}[section]
\newtheorem{lemma}[theorem]{Lemma}%[section]
\newtheorem{corollary}[theorem]{Corollary}%[section]
\newtheorem{conjecture}{Conjecture}%[section]
\newcommand{\blue}[1]{\textcolor{blue}{#1}}

\title{A project on cyclic ordering of some families of graphs\thanks{This is a summer research project (2022) of the three authors supervised by Dr. Xiaofeng Gu of University of West Georgia.}}
\author{Cedric Xia, Joseph Zhang, Allan Zhou
\thanks{The authors are listed in alphabetical order by last name.
\newline Email: xiacedric909@gmail.com, zh.joseph64@gmail.com, allanzhou777@gmail.com}
\\
\\
Westlake High School\\ Austin, TX 78746
}

\begin{document}
\date{\today}
\maketitle

\begin{abstract}
\normalsize
Let $G$ be an $n$-vertex connected graph. A cyclic base ordering of $G$ is a cyclic ordering of all edges such that every cyclically consecutive $n-1$ edges induce a spanning tree of $G$. In this project, we study cyclic base ordering of triangular grid graphs, series composition of graphs, generalized theta graphs, and circulant graphs.
\end{abstract}

%{\noindent {\bf MSC:} }

\newpage
\tableofcontents

\newpage
\section{Introduction}

Let $G=(V, E)$ be a connected graph.
\begin{definition}
A {\bf cyclic base ordering} or for short {\bf CBO} of $G$ is a cyclic ordering of $E(G)$ such that every cyclically consecutive $|V(G)|-1$ edges induce a spanning tree of $G$. Equivalently, a cyclic base ordering is a bijection $\mathcal{O}: E(G) \longrightarrow \{1,2,\ldots,|E(G)|\}$ such that for each $i=1,2,\ldots,|E(G)|$, $\{\mathcal{O}^{-1}(k): k= i, i+1,\ldots, i+|V(G)|-2\}$ induces a spanning tree of $G$, where the labelling $k$ is equivalent to $k-|E(G)|$ if $k >|E(G)|$. If $G$ has a CBO, then we say $G$ is {\bf cyclically orderable}.
\end{definition}

%\red{To Joseph: please add graph density, uniformly dense, the CBO conjecture here.}
\begin{definition}
The {\bf density} of a connected graph $G$ is defined to be $\displaystyle d(G)=\frac{|E(G)|}{|V(G)|-1}$, and $G$ is {\bf uniformly dense} if $d(H)\leq d(G)$ for every connected subgraph $H$ of $G$.
%Let $G$ be a graph with $E(G) \neq \emptyset$ and  $\omega(G)$ denote the number of components in $G$. The {\bf graph density} of $G$ and its subgraphs are $d(G)$ and $\gamma(G)$ respectively, defined as $d(G) = \frac{|E(G)|}{|V (G)|−\omega(G)}$ and $\gamma(G) = max\{d(H)\}$ where $H$ runs over all subgraphs of $G$ with $E(H) \neq \emptyset$. A graph G is {\bf{uniformly dense}} if $d(G) = \gamma(G)$.
\end{definition}

Cyclically orderable graphs are closely related to the density of graphs. Kajitani, Ueno and Miyano~\cite{KaUM88} made the following conjecture, and they proved the necessity.

\begin{conjecture}[Kajitani, Ueno and Miyano~\cite{KaUM88}]
A connected graph $G$ is cyclically orderable if and only if $G$ is uniformly dense.
\end{conjecture}

The sufficiency is still unsettled, but has been verified for many graph families in~\cite{KaUM88,GuHL14,GLYZ22,GZ22,Xiao22}.

In this project, we study cyclic base ordering of triangular grid graphs, series composition of graphs, generalized theta graphs, and circulant graphs, etc.

Given an edge ordering of a graph $G$, every cyclically consecutive $|V(G)|-1$ edges is called a {\bf progression}. To verify an edge ordering is a CBO, it suffices to show that any progression induces a spanning tree.

\section{Triangular grid graphs}

\begin{definition}
Let $k$ be a positive integer and $T_k$ be a triangular grid graph with $k$ levels which forms the shape of a larger triangle.
\end{definition}

Here are some examples of triangular grid graphs.
\begin{figure}[htbp]
    \centering
        \begin{tikzpicture}[every node/.style={circle,thick,draw},scale=1.3] 
        
        \begin{scope}
        \node (1) at (0, 0) {};
        \end{scope}
        
        \begin{scope}[xshift=1.5cm]
        \node (1) at (0.625, 1) {};
        \node (2) at (0, 0) {};
        \node (3) at (1.25, 0) {};
        \begin{scope}[>={},every node/.style={fill=white,circle,inner sep=0pt,minimum size=12pt}]
        \path [] (1) edge  (2);
        \path [] (1) edge  (3);
        \path [] (2) edge  (3);
        \end{scope}
        \end{scope}

        \begin{scope}[xshift=4.25cm]
        \node (1) at (1.25, 2) {};
        \node (2) at (0.625, 1) {};
        \node (3) at (1.875, 1) {};
        \node (4) at (0, 0) {};
        \node (5) at (1.25, 0) {};
        \node (6) at (2.5, 0) {};
        \begin{scope}[>={},every node/.style={fill=white,circle,inner sep=0pt,minimum size=12pt}]
        \path [] (1) edge  (2); 
        \path [] (1) edge  (3);
        \path [] (2) edge  (3);
        \path [] (2) edge  (4); 
        \path [] (2) edge  (5);
        \path [] (3) edge  (5);
        \path [] (3) edge  (6); 
        \path [] (4) edge  (5);
        \path [] (5) edge  (6);
        \end{scope}
        \end{scope}
        
        \begin{scope}[xshift=8.25cm] 
        \node (1) at (1.875, 3) {};
        \node (2) at (1.25, 2) {};
        \node (3) at (2.5, 2) {};
        \node (4) at (0.625, 1) {};
        \node (5) at (1.875, 1) {};
        \node (6) at (3.125, 1) {};
        \node (7) at (0, 0) {};
        \node (8) at (1.25, 0) {};
        \node (9) at (2.5, 0) {};
        \node (10) at (3.75, 0) {};
        \begin{scope}[>={},every node/.style={fill=white,circle,inner sep=0pt,minimum size=12pt}]
            \path [] (1) edge  (2);
            \path [] (1) edge  (3);
            \path [] (2) edge  (3);
            \path [] (2) edge  (4);
            \path [] (2) edge  (5);
            \path [] (3) edge  (5);
            \path [] (3) edge  (6);
            \path [] (4) edge  (5);
            \path [] (5) edge  (6);
            \path [] (4) edge  (7);
            \path [] (4) edge  (8);
            \path [] (5) edge  (8);
            \path [] (5) edge  (9);
            \path [] (6) edge  (9);
            \path [] (6) edge  (10);
            \path [] (7) edge  (8);
            \path [] (8) edge  (9);
            \path [] (9) edge  (10);
        \end{scope}
        \end{scope}

        \end{tikzpicture}
        \caption{Triangular grid graphs $T_1$, $T_2$, $T_3$, $T_4$}
\end{figure}
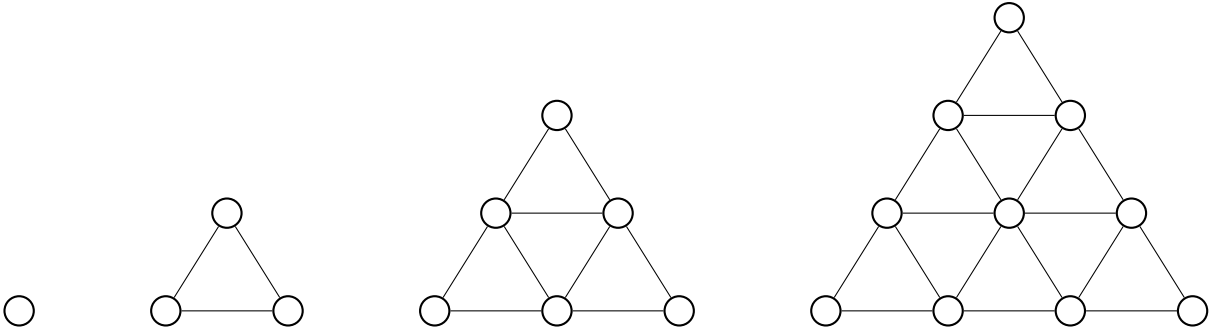

\begin{lemma}\label{lem:tk}
For any $T_k$, $\displaystyle |V(T_k)|=\frac{k(k+1)}{2}$, $\displaystyle |E(T_k)|=\frac{3k(k-1)}{2}$; and $\delta(T_k)=2$ if $k\ge 2$.
\end{lemma}
\begin{proof}
Observe that $|V(T_k)|$ is the $k$-th triangle number. Thus 
$$|V(T_{k})|=|V(T_{k-1})|+k$$ and so
$$|V(T_k)|=\frac{k(k+1)}{2}.$$

Further observe that each new level adds $2$ edges per vertex on the previous bottom row, and one edge between adjacent vertices on the new bottom row. It follows that 
$$|E(T_{k})|=|E(T_{k-1})|+3(k-1)$$
and
$$|E(T_k)|=\frac{3k(k-1)}{2}.$$
\end{proof}

\begin{lemma}[\cite{LYZ21}]\label{lem:mindegree}
If a nonempty graph $G$ has a CBO, then $\displaystyle \delta(G)\ge\frac{|E(G)|}{|V(G)|-1}$.
\end{lemma}

\begin{theorem}
A triangular grid graph $T_k$ is cyclically orderable if and only if $k\le 4$.
%If a triangular grid graph has a CBO, then $n\le4$.
\end{theorem}
\begin{proof}
For $k=1$, $T_k$ has no edge and it is trivial. For each $k=2,3,4$, a CBO of each $T_k$ is shown in Figure~\ref{fig:triangularCBO}.
%\red{(To Cedric: the CBO for $T_4$ seems not correct. Please verify it.)}
%comment: hopefully this works now. \red{It seems still not correct. we discovered one during the meeting, you may check to find it in the notes.}

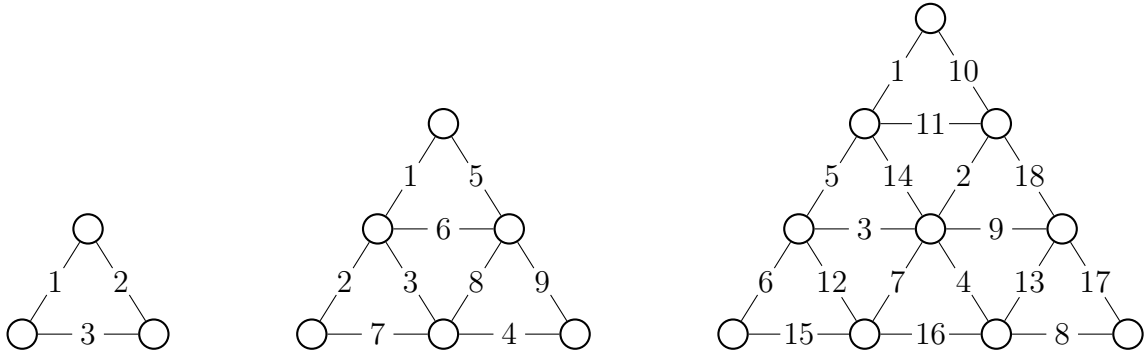
\begin{figure}[htbp]
    \centering
        \begin{tikzpicture}[every node/.style={circle,thick,draw},scale=1.4] 
%        \begin{scope}
%        \node (1) at (0, 0) {$v_1$};
%        \end{scope}
        
        \begin{scope}[xshift=1.5cm]
        \node (1) at (0.625, 1) {};
        \node (2) at (0, 0) {};
        \node (3) at (1.25, 0) {};
        \begin{scope}[>={},every node/.style={fill=white,circle,inner sep=0pt,minimum size=12pt}]
        \path [] (1) edge node {1} (2);
        \path [] (1) edge node {2} (3);
        \path [] (2) edge node {3} (3);
        \end{scope}
        \end{scope}

        \begin{scope}[xshift=4.25cm]
        \node (1) at (1.25, 2) {};
        \node (2) at (0.625, 1) {};
        \node (3) at (1.875, 1) {};
        \node (4) at (0, 0) {};
        \node (5) at (1.25, 0) {};
        \node (6) at (2.5, 0) {};
        \begin{scope}[>={},every node/.style={fill=white,circle,inner sep=0pt,minimum size=12pt}]
        \path [] (1) edge node {1} (2); 
        \path [] (1) edge node {5} (3);
        \path [] (2) edge node {6} (3);
        \path [] (2) edge node {2} (4); 
        \path [] (2) edge node {3} (5);
        \path [] (3) edge node {8} (5);
        \path [] (3) edge node {9} (6); 
        \path [] (4) edge node {7} (5);
        \path [] (5) edge node {4} (6);
        \end{scope}
        \end{scope}
        
        \begin{scope}[xshift=8.25cm]  % How do I make the nodes all slightly bigger, so they are the same size as v_10
        \node (1) at (1.875, 3) {};
        \node (2) at (1.25, 2) {};
        \node (3) at (2.5, 2) {};
        \node (4) at (0.625, 1) {};
        \node (5) at (1.875, 1) {};
        \node (6) at (3.125, 1) {};
        \node (7) at (0, 0) {};
        \node (8) at (1.25, 0) {};
        \node (9) at (2.5, 0) {};
        \node (10) at (3.75, 0) {};
        \begin{scope}[>={},every node/.style={fill=white,circle,inner sep=0pt,minimum size=12pt}]
            \path [] (1) edge node {1} (2);
            \path [] (1) edge node {10} (3);
            \path [] (2) edge node {11} (3);
            \path [] (2) edge node {5} (4);
            \path [] (2) edge node {14} (5);
            \path [] (3) edge node {2} (5);
            \path [] (3) edge node {18} (6);
            \path [] (4) edge node {3} (5);
            \path [] (5) edge node {9} (6);
            \path [] (4) edge node {6} (7);
            \path [] (4) edge node {12} (8);
            \path [] (5) edge node {7} (8);
            \path [] (5) edge node {4} (9);
            \path [] (6) edge node {13} (9);
            \path [] (6) edge node {17} (10);
            \path [] (7) edge node {15} (8);
            \path [] (8) edge node {16} (9);
            \path [] (9) edge node {8} (10);
        \end{scope}
        \end{scope}

        \end{tikzpicture}
        \caption{CBOs of $T_2$, $T_3$, $T_4$}
        \label{fig:triangularCBO}
\end{figure}

Now we show that $T_k$ is not cyclically orderable if $k> 4$. By Lemma~\ref{lem:tk},
$$\frac{|E(T_k)|}{|V(T_k)|-1}=\frac{\frac{3k(k-1)}{2}}{\frac{k(k+1)}{2}-1}=\frac{3k^2-3k}{k^2+k-2}.$$
It is not hard to verify that $\displaystyle \frac{|E(T_k)|}{|V(T_k)|-1}=\frac{3k^2-3k}{k^2+k-2} > 2 =\delta(T_k)$ when $k>4$, contrary to Lemma~\ref{lem:mindegree}.
%Observe that for all triangular grid graphs $T_k$ where $k\ge 2$, $\delta(T_k)=2$. By Lemma~\ref{lem:mindegree}, $$2\ge\frac{|E(T)|}{|V(T)|-1}=\frac{\frac{3n(n-1)}{2}}{\frac{n(n+1)}{2}-1}=\frac{3n^2-3n}{n^2+n-2}.$$ Solving for n yields $$n^2-5n+4\le0\longrightarrow n\le4$$
\end{proof}

\section{Series composition of graphs}
\begin{definition}
Given graphs $G$ and $H$ with $u\in V(G)$ and $v\in V(H)$. The {\bf series composition}  of $G$ and $H$, denoted $G\oplus H$, is the graph obtained from $G$ and $H$ by gluing $u$ and $v$ as a single vertex. Note that $G\oplus H$ may not be unique if $u,v$ are not specified. A series composition is also called a {\bf 1-sum}.
%More rigorously, let $v_G$ and $v_H$ be vertexes in $G$ and $H$ respectively. Then $G \oplus H$ can be generated by replacing $v_H$ with $v_G$, preserving all edges to $v_G$ and transferring all edges from $v_H$ to $v_G$. Note that there can be multiple versions of $G \oplus H$. 
\end{definition}

Here is an example of series composition, which is actually a friendship graph. A {\bf friendship graph} $Fd_t$ is a graph obtained from $t$ copies of $K_3$ by gluing them at a single vertex. Generally, a {\bf windmill graph} $Wd(k,t)$ is a a graph obtained from $t$ copies of the complete graph $K_k$ by gluing them at a single vertex. Clearly, $Fd_t=Wd(3,t)$.

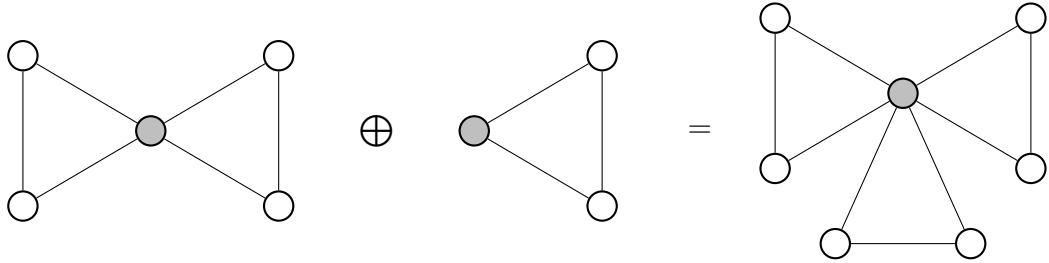
\begin{figure}[htb]
    \centering
    \begin{tikzpicture}[every node/.style={circle,thick,draw},scale=1] 
        \begin{scope}
            \node (1) at (0, 0) {};
            \node (2) at (0, 2) {};
            \node[fill=lightgray] (3) at (1.7, 1) {};
            \node (4) at (3.4, 0) {};
            \node (5) at (3.4, 2) {};
            \begin{scope}[>={},every node/.style={fill=white,circle,inner sep=0pt,minimum size=12pt}]
                \path [] (1) edge (2);
                \path [] (1) edge (3);
                \path [] (2) edge (3);
                \path [] (3) edge (4);
                \path [] (3) edge (5);
                \path [] (4) edge (5);
            \end{scope}
        \end{scope}
    \node[draw=white,circle,inner sep=0pt,minimum size=12pt] at (4.7, 1) {$\bigoplus$};
        \begin{scope}[xshift = 6cm]
            \node[fill=lightgray] (1) at (0, 1) {};
            \node (2) at (1.7, 2) {};
            \node (3) at (1.7, 0) {};
            \begin{scope}[>={},every node/.style={fill=white,circle,inner sep=0pt,minimum size=12pt}]
                \path [] (1) edge  (2);
                \path [] (1) edge  (3);
                \path [] (2) edge  (3);
            \end{scope}
        \end{scope}
    \node[draw=white,circle,inner sep=0pt,minimum size=12pt] at (9, 1) {$=$};

    \begin{scope}[xshift = 10cm, yshift=0.5cm]
            \node (1) at (0, 0) {};
            \node (2) at (0, 2) {};
            \node[fill=lightgray] (3) at (1.7, 1) {};
            \node (4) at (3.4, 0) {};
            \node (5) at (3.4, 2) {};
            \node (6) at (0.8, -1) {};
            \node (7) at (2.6, -1) {};
            \begin{scope}[>={},every node/.style={fill=white,circle,inner sep=0pt,minimum size=12pt}]
                \path [] (1) edge  (2);
                \path [] (1) edge  (3);
                \path [] (2) edge  (3);
                \path [] (3) edge  (4);
                \path [] (3) edge  (5);
                \path [] (4) edge  (5);
                \path [] (3) edge  (6);
                \path [] (3) edge  (7);
                \path [] (6) edge  (7);
            \end{scope}
        \end{scope}
    \end{tikzpicture}
    \caption{An example of $G\oplus H$}
%   \label{}
\end{figure}

\begin{lemma}\label{lem:sctree}
Any series composition of two trees is also a tree.
\end{lemma}

\begin{theorem}\label{thm:series}
Let $G$ and $H$ be two cyclically orderable graphs with $|V(G)|=|V(H)|$ and $ |E(G)|=|E(H)|$. Then $G \oplus H$ is also cyclically orderable.
%with CBOs $e_1, e_2, e_3, \dots e_n$ and $e_1', e_2', e_3', \dots e_n'$, respectively. Then $G \oplus H$ has a CBO in the form \[e_1, e_1', e_2, e_2', \dots, e_n, e_n'.\] 
\end{theorem}
\begin{proof}
Let $\mathcal{O}_G$ and $\mathcal{O}_H$ be a CBO of $G$ and $H$, respectively. Suppose that $|V(G)|=|V(H)|=n$ and $ |E(G)|=|E(H)|=m$. Then $G \oplus H$ has exactly $2n-1$ vertices and $2m$ edges. We define an edge ordering $\mathcal{O}$ of $G \oplus H$ by
\begin{align*}
\mathcal{O}(e) = \begin{cases}
2\mathcal{O}_G(e) -1 & \text{ if } e \in E(G) \\
2\mathcal{O}_H(e) & \text{ if } e \in E(H).
\end{cases}
\end{align*}
Roughly speaking, the edges in $\mathcal{O}$ are alternatively from $\mathcal{O}_G$ and $\mathcal{O}_H$.

To see $\mathcal{O}$ is a CBO of $G\oplus H$, it suffices to show that each progression of $\mathcal{O}$ induces a spanning tree of $G\oplus H$. Let $P$ be a progression of $\mathcal{O}$. Then $P$ contains exactly $2n-2$ edges. These edges come from $n-1$ consecutive edges of $\mathcal{O}_G$ and $n-1$ consecutive edges of $\mathcal{O}_H$, and thus induce a spanning tree $T_1$ of $G$ and a spanning tree $T_2$ of $H$. Clearly $P$ induces $T_1\oplus T_2$. By Lemma~\ref{lem:sctree}, $T_1\oplus T_2$ is a spanning tree of $G\oplus H$, completing the proof.
%Let $e_1, e_1', e_2, e_2', \dots, e_n, e_n'$ be an ordering of $G \oplus H$. Since all edges are conserved in series composition, it follows that $|E(G \oplus H)| = |E(G)| + |E(H)|$. All nodes are conserved, except the one removed, so $|V(G \oplus H)| = |V(G)| + |V(H)| - 1$ and each progression of $G \oplus H$ has $|V(G)| + |V(H)| -2 = 2(|V(G)|-1)$ edges. Thus each progression of the ordering takes the form $$e_i, e_i', e_{i+1}, e_{i+1}', \dots e_{i+|V(G)|-2}, e_{i+|V(G)|-2}'$$
\end{proof}

With a similar argument, Theorem~\ref{thm:series} can be generalized to series compositions of more graphs.
\begin{theorem}
Let $t\ge 2$ be an integer and $G_1, G_2,\cdots, G_t$ be cyclically orderable graphs with $|V(G_1)|= |V(G_2)|=\cdots=|V(G_t)|$ and $|E(G_1)|= |E(G_2)|=\cdots=|E(G_t)|$. Then $G_1\oplus G_2\oplus\cdots\oplus G_t$ is also cyclically orderable.
\end{theorem}

In fact, the above graph can be generalized even more.

\begin{proposition}
For two graphs $G$ and $H$, if $|d(G)|=|d(H)|=d$, then $d(G\oplus H) =d$.
\end{proposition}
\begin{proof}
Suppose that $G$ has $m_1$ edges and $n_1$ vertices, while $H$  has $m_2$ edges and $n_2$ vertices, and  $|d(G)|=|d(H)| =d$. Then $d(G) =m_1/(n_1-1) = d$, and so $m_1 = d(n_1-1)$. Similarly, $m_2 = d(n_2-1)$.

Notice that $G\oplus H$ has exactly $m_1 +m_2$ edges and $n_1 +n_2 -1$ vertices. It follows that 
$$d(G\oplus H) =\frac{m_1 +m_2}{n_1 +n_2 -2} =\frac{d(n_1-1) +d(n_2-1)}{n_1 +n_2 -2} = \frac{d(n_1 +n_2 -2)}{n_1 +n_2 -2} =d.$$
\end{proof}

\begin{theorem}
Let $G$ and $H$ be two cyclically orderable graphs with $|d(G)|=|d(H)|$. Then $G\oplus H$ is also cyclically orderable.
\end{theorem}
\begin{proof}
Suppose that $G$ has $m_1$ edges and $n_1$ vertices, while $H$  has $m_2$ edges and $n_2$ vertices, and  $|d(G)|=|d(H)| =d$. Let $d = s/t$ for integers $s, t$ with $\gcd(s,t)=1$. 

Then $$d(G) =\frac{m_1}{n_1-1} = \frac{s}{t},$$
and so $m_1 t = s(n_1-1)$, which implies that $t\mid s(n_1-1)$. Since $\gcd(s,t)=1$, it follows that $t\mid (n_1-1)$. Let $n_1-1 = t p_1$. Then $m_1 = s\cdot \frac{n_1-1}{t} = s p_1$.

Similarly, $m_2 t = s(n_2-1)$ and $t\mid (n_2-1)$. Let $n_2 -1 = t p_2$. Then $m_2 = s\cdot \frac{n_2 -1}{t} = s p_2$.

Let $\mathcal{O}_G$ and $\mathcal{O}_H$ be a CBO of $G$ and $H$, respectively. Since $m_1 = s p_1$, $\mathcal{O}_G$ can be divided into $s$ ordered subsets $P_1,P_2,\cdots,P_s$, each with size $p_1$. That is 
$$\mathcal{O}_G = P_1, P_2,\cdots, P_s.$$
Similarly, since $m_2 = s p_2$, $\mathcal{O}_H$ can be divided into $s$ ordered subsets $Q_1,Q_2,\cdots,Q_s$, each with size $p_2$. That is 
$$\mathcal{O}_H = Q_1, Q_2,\cdots, Q_s.$$

Now we define an edge ordering $\mathcal{O}$ for $G\oplus H$ by
$$\mathcal{O} = P_1, Q_1, P_2, Q_2,\cdots, P_s, Q_s.$$

Notice that $G\oplus H$ has exactly $m_1 +m_2$ edges and $n_1 +n_2 -1$ vertices. Any progression of $\mathcal{O}$ has exactly $n_1 +n_2 -2$ edges. Also notice that $n_1 +n_2 -2 = n_1 -1 + n_2 -1 = t p_1 + t p_2$, and edges in $\mathcal{O}$ are alternately from $P_i$ and $Q_i$. It follows that any progression of $\mathcal{O}$ must have exactly $n_1-1$ consecutive edges (inducing a spanning tree $T_1$ of $G$) from $\mathcal{O}_G$ and exactly $n_2-1$ consecutive edges (inducing a spanning tree $T_2$ of $H$) from $\mathcal{O}_H$. By Lemma~\ref{lem:sctree}, $T_1\oplus T_2$ is a spanning tree of $G\oplus H$, completing the proof.
\end{proof}

\begin{corollary}
Let $t\ge 2$ be an integer and $G_1, G_2,\cdots, G_t$ be cyclically orderable graphs with $d(G_1)= d(G_2)=\cdots=d(G_t)$. Then $G_1\oplus G_2\oplus\cdots\oplus G_t$ is also cyclically orderable.
\end{corollary}

Thus friendship graphs and windmill graphs are cyclically orderable.

\section{Generalized theta graphs}

\begin{definition}
A generalized theta graph, denoted by $\Theta_{l_1, l_2,\ldots, l_k}$ is a graph obtained by joining two distinct vertices through $k$ disjoint paths of lengths $l_1,l_2,\ldots,l_k$ with $l_1 \le l_2 \le \cdots \le l_k$.
\end{definition}
%Shown below is a graph of $\theta_{l_1, l_2,l_3}$ with three distinct paths of length $l_1, l_2, l_3$ between the two nodes.

Here is an example of generalized theta graphs.
\begin{figure}[htb]
    \centering
    \begin{tikzpicture}[every node/.style={circle,thick,draw},scale=1.6] 
        \begin{scope}
            \node[fill=lightgray] (1) at (0,1) {};
            \node (2) at (1, 2) {};
%            \node (3) at (2, 2) {};
            \node (4) at (3, 2) {};
            \node[fill=lightgray] (5) at (4, 1) {};
            
            \node (6) at (0,1) {};
            \node (7) at (1, 1) {};
            \node (8) at (2, 1) {};
            \node (9) at (3, 1) {};
            \node (10) at (4, 1) {};
            
            \node (11) at (0,1) {};
            \node (12) at (1, 0) {};
            \node (13) at (2, 0) {};
            \node (14) at (3, 0) {};
            \node (15) at (4, 1) {};
        \begin{scope}[>={},every node/.style={fill=white,circle,inner sep=0pt,minimum size=10pt}]
            \path [] (1) edge   (2);
            \path [] (2) edge   (4);
%            \path [] (3) edge   (4);
            \path [] (4) edge   (5);
            \path [] (6) edge   (7);
            \path [] (7) edge   (8);
            \path [] (8) edge   (9);
            \path [] (9) edge   (10);
            \path [] (11) edge   (12);
            \path [] (12) edge   (13);
            \path [] (13) edge   (14);
            \path [] (14) edge   (15);

            \end{scope}
        \end{scope}
    \end{tikzpicture}
    \caption{The generalized theta graph $\Theta_{3,4,4}$}
%   \label{}
\end{figure}
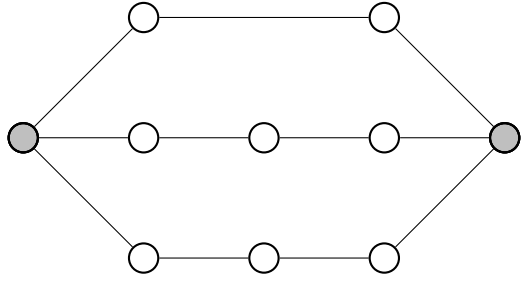

\begin{lemma}\label{lem:theta}
Let $G$ be the generalized theta graph $\Theta_{l_1,l_2,\ldots,l_k}$. Then
$$|E(G)| = l_1 +l_2 +\cdots + l_k = \sum_{i=1}^{k} l_i$$ and $$|V(G)| =\sum_{i=1}^{k} (l_i-1) + 2 = \sum_{i=1}^{k} l_i - k + 2.$$
\end{lemma}

Not all generalized theta graphs are cyclically orderable. Here is a necessary condition.
\begin{theorem}
If the generalized theta graph $\Theta_{l_1, l_2,\ldots,l_k}$ is cyclically orderable, then
$$\frac{\sum_{i=1}^{t} l_i}{t-1} \ge \frac{\sum_{i=1}^{k} l_i}{k-1},$$
for any integer t with $2\le t\le k$.
\end{theorem}
\begin{proof}
Let $G$ be the generalized theta graph $\Theta_{l_1, l_2, \ldots, l_k}$ and $H$ be a subgraph of $G$ consisting of $t$ disjoint paths of length $l_1, l_2,\ldots,l_t$. If $G$ is cyclically orderable, then G is uniformly dense, and so $d(H)\le d(G)$.
%We recall from intrinsic properties of a CBO that $\frac{|E(G)}{|V(G)|-1} \ge $ degree of every vertex. 

By Lemma~\ref{lem:theta}, we have
$$d(G)= \frac{|E{G}|}{|V(G)|-1} =\frac{\sum_{i=1}^{k} l_i}{\sum_{i=1}^{k} l_i - k + 1} = \frac{1}{1-\frac{k-1}{\sum_{i=1}^{k} l_i}}.$$
Similarly, we have
$$d(H)= \frac{|E{G}|}{|V(G)|-1} =\frac{\sum_{i=1}^{t} l_i}{\sum_{i=1}^{t} l_i - t + 1} = \frac{1}{1-\frac{t-1}{\sum_{i=1}^{t} l_i}}.$$
%Since $$d(H) \le d(G)\rightarrow\frac{t-1}{\frac{t}{2}} l_i \le \frac{k-1}{\sum_{i=1}^{k} l_i}$$

Since $d(H)\le d(G)$, it follows that for any $2\le t\le k$,
$$\frac{\sum_{i=1}^{t} l_i}{t-1} \ge \frac{\sum_{i=1}^{k} l_i}{k-1}.$$    
\end{proof}

For instance, $\Theta_{1,2,5}$ is not cyclically orderable.
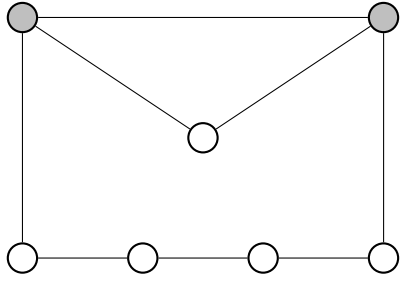
\begin{figure}[htb]
    \centering
    \begin{tikzpicture}[every node/.style={circle,thick,draw},scale=1.6] 
        \begin{scope}
            \node[fill=lightgray] (1) at (0,2) {};
            \node[fill=lightgray] (2) at (3, 2) {};
            \node (3) at (1.5, 1) {};
            \node (4) at (0, 0) {};
            \node (5) at (1, 0) {};
            
            \node (6) at (2,0) {};
            \node (7) at (3,0) {};
        \begin{scope}[>={},every node/.style={fill=white,circle,inner sep=0pt,minimum size=10pt}]
            \path [] (1) edge  (2);
            \path [] (1) edge  (3);
            \path [] (2) edge  (3);
            \path [] (1) edge  (4);
            \path [] (4) edge  (5);
            \path [] (5) edge  (6);
            \path [] (6) edge  (7);
            \path [] (7) edge  (2);

            \end{scope}
        \end{scope}
    \end{tikzpicture}
    \caption{The generalized theta graph $\Theta_{1,2,5}$}
%   \label{}
\end{figure}
%$$3=\frac{3}{2-1} < \frac{8}{3-1}=4$$

\begin{theorem}
If $l_1 =l_2 =\cdots =l_k$, then $\Theta_{l_1, l_2, \ldots, l_k}$ is cyclically orderable.
\end{theorem}
\begin{proof}
Suppose that $l_1 =l_2 =\cdots =l_k =l$. Then $\Theta_{l_1, l_2, \ldots, l_k}$ has $kl$ edges and $kl - k+2$ vertices. Suppose the $k$ disjoint paths are $P_1, P_2,\ldots, P_k$. For $i=1,2,\cdots,k$, suppose the edges on the path $P_i$ are $e_{i1}, e_{i2},\ldots, e_{il}$.

We can define an edge ordering $\mathcal{O}$ for $\Theta_{l_1, l_2, \ldots, l_k}$ by $$\mathcal{O}(e_{ij}) = (j-1)k +i,$$ where $1\le i\le k$ and $1\le j\le l$. Here is an example when $k=3$ and $l=5$.

It is not hard to verify $\mathcal{O}$ is a cyclic base ordering of $\Theta_{l_1, l_2, \ldots, l_k}$.

\begin{figure}[htb]
    \centering
    \begin{tikzpicture}[every node/.style={circle,thick,draw},scale=2] 
        \begin{scope}
            \node[fill=lightgray] (1) at (0,1) {};
            \node (2) at (1, 2) {};
            \node (3) at (2, 2) {};
            \node (4a) at (3, 2) {};
            \node (4) at (4, 2) {};
            \node[fill=lightgray] (5) at (5, 1) {};
            
            \node (6) at (0,1) {};
            \node (7) at (1, 1) {};
            \node (8) at (2, 1) {};
            \node (9a) at (3, 1) {};
            \node (9) at (4, 1) {};
            \node (10) at (5, 1) {};
            
            \node (11) at (0,1) {};
            \node (12) at (1, 0) {};
            \node (13) at (2, 0) {};
            \node (14a) at (3, 0) {};
            \node (14) at (4, 0) {};
            \node (15) at (5, 1) {};
        \begin{scope}[>={},every node/.style={fill=white,circle,inner sep=0pt,minimum size=10pt}]
            \path [] (1) edge node {1} (2);
            \path [] (2) edge node {4} (3);
            \path [] (3) edge node {7} (4a);
            \path [] (4a) edge node {10} (4);
            \path [] (4) edge node {13} (5);
            \path [] (6) edge node {2} (7);
            \path [] (7) edge node {5} (8);
            \path [] (8) edge node {8}(9a);
            \path [] (9a) edge node {11} (9);
            \path [] (9) edge node {14} (10);
            \path [] (11) edge node {3} (12);
            \path [] (12) edge node {6} (13);
            \path [] (13) edge node {9} (14a);
            \path [] (14a) edge node {12} (14);
            \path [] (14) edge node {15} (15);
%            \path [dotted] (0.5,0.5) edge (4.5,.5);

            \end{scope}
        \end{scope}
    \end{tikzpicture}
    \caption{CBO of $\Theta_{5,5,5}$}
%   \label{}
\end{figure}
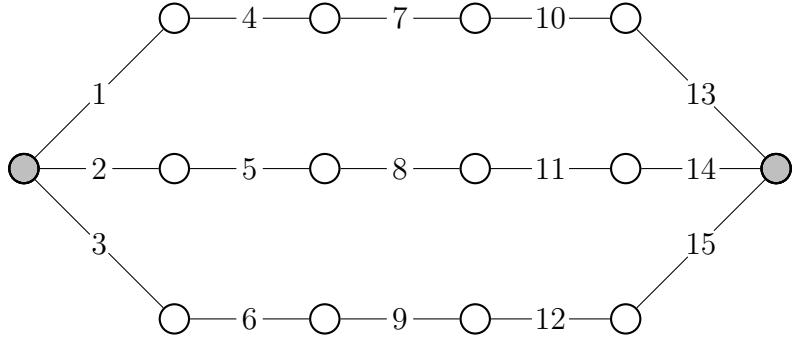
\end{proof}

\section{Circulant graphs}

\begin{definition}
Let integers $x_1,x_2,\ldots, x_t < n/2$. A {\bf{circulant graph}} on $n$ vertices, denoted by $Ci_{n}(x_1, x_2, \ldots, x_t)$, is a graph with the vertex set $\{v_1,v_2,\ldots, v_n\}$ and each vertex $v_i$ is joined to $v_j$ where $j= i+ x_k$ (the subscripts are equivalent modulo $k$ if $j>k$), for every $k=1,2,\cdots,t$.
%obtained from a cycle on $n$ vertices with an edge connecting every $x_i$th vertex. 
\end{definition}

Here are two examples.
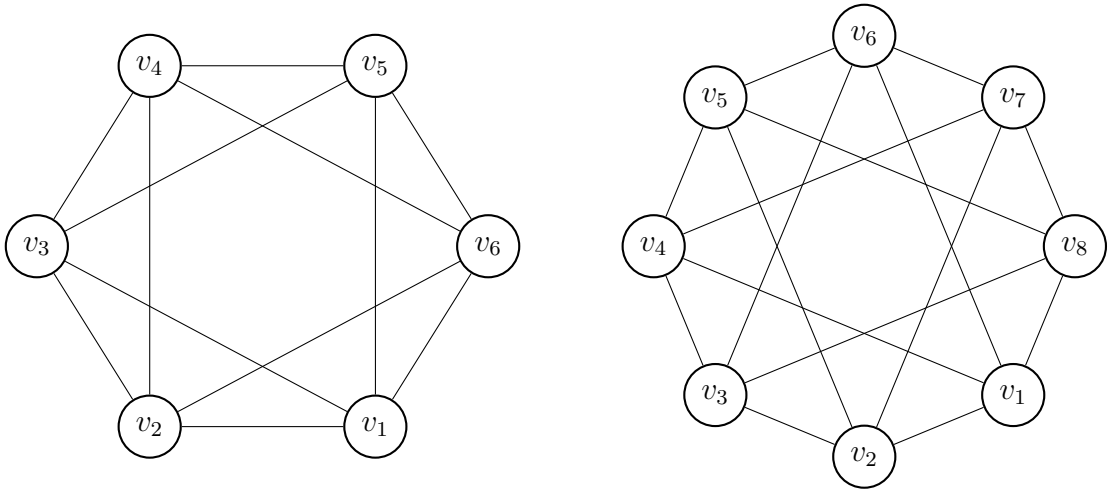
\begin{figure}[htb]
\centering{
\begin{tikzpicture}[every node/.style={circle,thick,draw}] 

\begin{scope}[scale=1.2]
\node (1) at (3.75, 0.5) {$v_1$};
\node (2) at (1.25, 0.5) {$v_2$};
\node (3) at (0, 2.5) {$v_3$};
\node (4) at (1.25, 4.5) {$v_4$};
\node (5) at (3.75, 4.5) {$v_5$};
\node (6) at (5, 2.5) {$v_6$};
\begin{scope}[>={},every node/.style={fill=white,circle,inner sep=0pt,minimum size=12pt}]
  \path [] (1) edge (3);
  \path [] (4) edge (5);
  \path [] (2) edge (4);
  \path [] (5) edge (6);
  \path [] (3) edge (5);
  \path [] (1) edge (6);
  \path [] (4) edge (6);
  \path [] (1) edge (2);
  \path [] (1) edge (5);
  \path [] (2) edge (3);
  \path [] (2) edge (6);
  \path [] (3) edge (4);
\end{scope}
\end{scope}

\begin{scope}[xshift=11cm,yshift=3cm,scale=1.4]
  \foreach \i in {1,2,3,4,5,6,7,8} {% 8 is the number of vertices
    \node (\i) at (-360/8 *\i:2cm) [circle,thick,draw,fill=white] {$v_{\i}$};
  }
  \begin{scope}[>={},every node/.style={fill=white,circle,inner sep=0pt,minimum size=12pt}]
    \path (1) edge (2);
    \path (2) edge (3);
    \path (3) edge (4);
    \path (4) edge (5);
    \path (5) edge (6);
    \path (6) edge (7);
    \path (7) edge (8);
    \path (8) edge (1);
    
    \path (1) edge (4);
    \path (2) edge (5);
    \path (3) edge (6);
    \path (4) edge (7);
    \path (5) edge (8);
    \path (6) edge (1);
    \path (7) edge (2);
    \path (8) edge (3);
  \end{scope}
\end{scope}
\end{tikzpicture}
}
\caption{$Ci_6(1,2)$ and $Ci_8(1,3)$}
%\label{}
\end{figure}

The graph $Ci_n(1,2)$ is actually {\bf the square} of cycle $C_n$, which is also denoted by $C_n^2$. %Here is another example of circulant graphs.
The CBOs of the circulant graphs $Ci_n(1,2)$ and $Ci_n(1,3)$ have been previously studied in \cite{LYZ21} and \cite{Xiao22}, respectively.

Here are some example of $Ci_n(1,4)$.

\begin{figure}[htb!]
\centering{
\begin{tikzpicture}[scale=1.8]
\begin{scope}
  \foreach \i in {1,2,3,4,5,6,7,8,9} {% 8 is the number of vertices
    \node (\i) at (-360/9 *\i:2cm) [circle,thick,draw,fill=white] {$v_{\i}$};
  }
  \begin{scope}[>={},every node/.style={fill=white,circle,inner sep=0pt,minimum size=12pt}]
    \path (1) edge node {10} (2);
    \path (2) edge node {12} (3);
    \path (3) edge node {14} (4);
    \path (4) edge node {16} (5);
    \path (5) edge node {18} (6);
    \path (6) edge node {2} (7);
    \path (7) edge node {4} (8);
    \path (8) edge node {6} (9);
    \path (9) edge node {8} (1);
    
    \path (1) edge[bend left=25] node {1} (5);
    \path (2) edge[bend left=25] node {3} (6);
    \path (3) edge[bend left=25] node {5} (7);
    \path (4) edge[bend left=25] node {7} (8);
    \path (5) edge[bend left=25] node {9} (9);
    \path (6) edge[bend left=25] node {11} (1);
    \path (7) edge[bend left=25] node {13} (2);
    \path (8) edge[bend left=25] node {15} (3);
    \path (9) edge[bend left=25] node {17} (4);
  \end{scope}
  \end{scope}
\end{tikzpicture}
}
\caption{$Ci_9(1,4)$}
%\label{}
\end{figure}
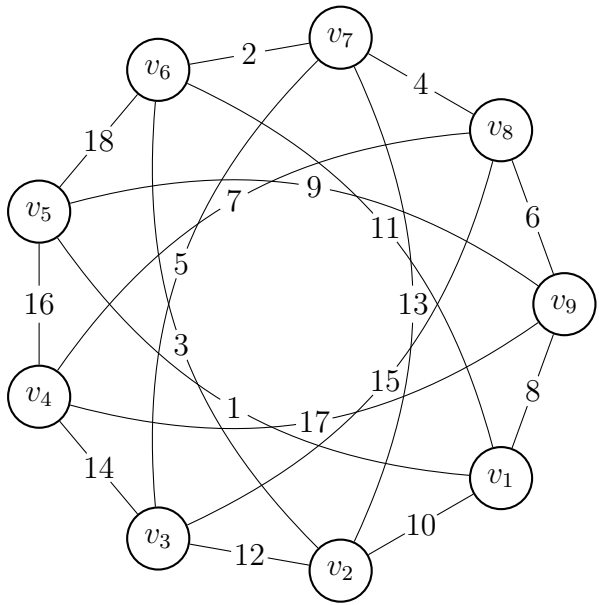

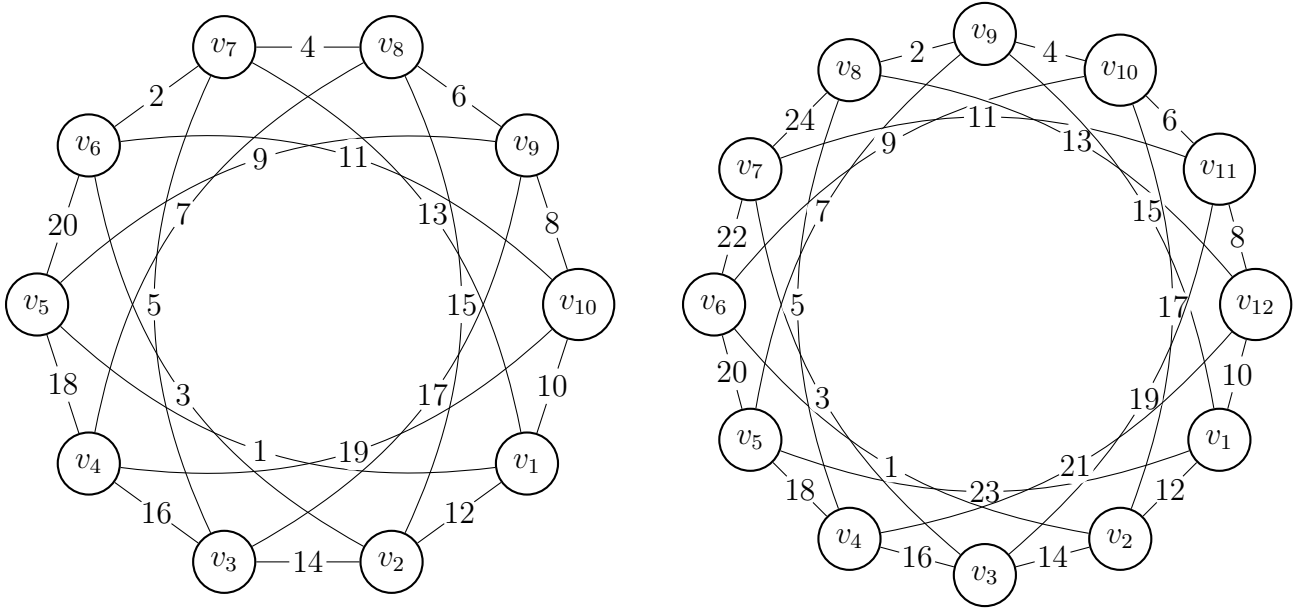
\begin{figure}[htb!]
\centering{
\begin{tikzpicture}[scale=1.8]
\begin{scope}G
  \foreach \i in {1,2,3,4,5,6,7,8,9,10} {% 8 is the number of vertices
    \node (\i) at (-360/10 *\i:2cm) [circle,thick,draw,fill=white] {$v_{\i}$};
  }
  \begin{scope}[>={},every node/.style={fill=white,circle,inner sep=0pt,minimum size=12pt}]
    \path (1) edge node {12} (2);
    \path (2) edge node {14} (3);
    \path (3) edge node {16} (4);
    \path (4) edge node {18} (5);
    \path (5) edge node {20} (6);
    \path (6) edge node {2} (7);
    \path (7) edge node {4} (8);
    \path (8) edge node {6} (9);
    \path (9) edge node {8} (10);
    \path (10) edge node {10} (1);
    
    \path (1) edge[bend left=25] node {1} (5);
    \path (2) edge[bend left=25] node {3} (6);
    \path (3) edge[bend left=25] node {5} (7);
    \path (4) edge[bend left=25] node {7} (8);
    \path (5) edge[bend left=25] node {9} (9);
    \path (6) edge[bend left=25] node {11} (10);
    \path (7) edge[bend left=25] node {13} (1);
    \path (8) edge[bend left=25] node {15} (2);
    \path (9) edge[bend left=25] node {17} (3);
    \path (10) edge[bend left=25] node {19} (4);
  \end{scope}
  \end{scope}
  \begin{scope}[xshift=5cm, scale=1]
  \foreach \i in {1,2,3,4,5,6,7,8,9,10,11,12} {
    \node (\i) at (-360/12 *\i:2cm) [circle,thick,draw,fill=white] {$v_{\i}$};
  }
  \begin{scope}[>={},every node/.style={fill=white,circle,inner sep=0pt,minimum size=12pt}]
    \path (1) edge node {12} (2);
    \path (2) edge node {14} (3);
    \path (3) edge node {16} (4);
    \path (4) edge node {18} (5);
    \path (5) edge node {20} (6);
    \path (6) edge node {22} (7);
    \path (7) edge node {24} (8);
    \path (8) edge node {2} (9);
    \path (9) edge node {4} (10);
    \path (10) edge node {6} (11);
    \path (11) edge node {8} (12);
    \path (12) edge node {10} (1);
    
    \path (1) edge[bend left=20] node {23} (5);
    \path (2) edge[bend left=20] node {1} (6);
    \path (3) edge[bend left=20] node {3} (7);
    \path (4) edge[bend left=20] node {5} (8);
    \path (5) edge[bend left=20] node {7} (9);
    \path (6) edge[bend left=20] node {9} (10);
    \path (7) edge[bend left=20] node {11} (11);
    \path (8) edge[bend left=20] node {13} (12);
    \path (9) edge[bend left=20] node {15} (1);
    \path (10) edge[bend left=20] node {17} (2);
    \path (11) edge[bend left=19] node {19} (3);
    \path (12) edge[bend left=19] node {21} (4);
  \end{scope}
  \end{scope}
\end{tikzpicture}
}
\caption{$Ci_{10}(1,4)$ and $Ci_{12}(1,4)$}
%\label{}
\end{figure}

\iffalse
\begin{theorem}
\blue{The circulant graph $Ci_n(1,4)$ is cyclically orderable.}
\end{theorem}
\begin{proof}
%For each $i=1,2,\ldots, n$, let $a_i$ be the edge joining $v_i$ and $v_{i+4}$, and $b_i$ be the edge joining $v_i$ and $v_{i+1}$. If a subscript $j$ is greater than $n$, then the subscript is $j-n$.
\end{proof}
\fi

\newpage
\section{1D Polygon Graphs}
%Algorithm:
Arbitrarily choose one of the non-shared edges of the first polygon which does not correspond to a shared edge, and number it $1$. Number the corresponding sides on each polygon sequentially, from left to right. Repeat until all non-shared edges are numbered, except the $2$ sides on each end polygon that correspond to shared edges. From left to right, number these sequentially. We claim this is a CBO.

\begin{figure}[htbp]
    \centering
        \begin{tikzpicture}[every node/.style={circle,thick,draw},scale=1.3] 
        
        \begin{scope}
        \node (1) at (0.5, 0.875) {$v_1$};
        \node (2) at (0, 2) {$v_2$};
        \node (3) at (1.125, 2.875) {$v_3$};
        \node (4) at (2.25, 2) {$v_4$};
        \node (5) at (1.75, 0.875) {$v_5$};
        \node (6) at (2.875, 0) {$v_6$};
        \node (7) at (4, 0.875) {$v_7$};
        \node (8) at (3.5, 2) {$v_8$};
        \node (9) at (4.625, 2.875) {$v_9$};
        \node (10) at (5.75, 2) {$v_{10}$};
        \node (11) at (5.25, 0.875) {$v_{11}$};
        \node (12) at (6.375, 0) {$v_{12}$};
        \node (13) at (7.5, 0.875) {$v_{13}$};
        \node (14) at (7, 2) {$v_{14}$};
        \node (15) at (8.125, 2.875) {$v_{15}$};
        \node (16) at (9.25, 2) {$v_{16}$};
        \node (17) at (8.75, 0.875) {$v_{17}$};
        \begin{scope}[>={},every node/.style={fill=white,circle,inner sep=0pt,minimum size=12pt}]
        \path [] (1) edge (2);
        \path [] (2) edge (3);
        \path [] (3) edge (4);
        \path [] (4) edge (5);
        \path [] (1) edge (5);
        \path [] (5) edge (6);
        \path [] (6) edge (7);
        \path [] (7) edge (8);
        \path [] (8) edge (4);
        \path [] (8) edge (9);
        \path [] (9) edge (10);
        \path [] (10) edge (11);
        \path [] (11) edge (7);
        \path [] (11) edge (12);
        \path [] (12) edge (13);
        \path [] (13) edge (14);
        \path [] (14) edge (10);
        \path [] (14) edge (15);
        \path [] (15) edge (16);
        \path [] (16) edge (17);
        \path [] (17) edge (13);
        \end{scope}
        \end{scope}
        
        \begin{scope}[yshift=-4cm]
        \node (1) at (0.5, 0.875) {$v_1$};
        \node (2) at (0, 2) {$v_2$};
        \node (3) at (1.125, 2.875) {$v_3$};
        \node (4) at (2.25, 2) {$v_4$};
        \node (5) at (1.75, 0.875) {$v_5$};
        \node (6) at (2.875, 0) {$v_6$};
        \node (7) at (4, 0.875) {$v_7$};
        \node (8) at (3.5, 2) {$v_8$};
        \node (9) at (4.625, 2.875) {$v_9$};
        \node (10) at (5.75, 2) {$v_{10}$};
        \node (11) at (5.25, 0.875) {$v_{11}$};
        \node (12) at (6.375, 0) {$v_{12}$};
        \node (13) at (7.5, 0.875) {$v_{13}$};
        \node (14) at (7, 2) {$v_{14}$};
        \node (15) at (8.125, 2.875) {$v_{15}$};
        \node (16) at (9.25, 2) {$v_{16}$};
        \node (17) at (8.75, 0.875) {$v_{17}$};
        \begin{scope}[>={},every node/.style={fill=white,circle,inner sep=0pt,minimum size=12pt}]
        \path [] (1) edge (2);
        \path [] (2) edge node {1} (3);
        \path [] (3) edge (4);
        \path [] (4) edge (5);
        \path [] (1) edge (5);
        \path [] (5) edge node {2} (6);
        \path [] (6) edge (7);
        \path [] (7) edge (8);
        \path [] (8) edge (4);
        \path [] (8) edge node {3} (9);
        \path [] (9) edge (10);
        \path [] (10) edge (11);
        \path [] (11) edge (7);
        \path [] (11) edge node {4} (12);
        \path [] (12) edge (13);
        \path [] (13) edge (14);
        \path [] (14) edge (10);
        \path [] (14) edge node {5} (15);
        \path [] (15) edge (16);
        \path [] (16) edge (17);
        \path [] (17) edge (13);
        \end{scope}
        \end{scope}
        
        \begin{scope}[yshift=-8cm]
        \node (1) at (0.5, 0.875) {$v_1$};
        \node (2) at (0, 2) {$v_2$};
        \node (3) at (1.125, 2.875) {$v_3$};
        \node (4) at (2.25, 2) {$v_4$};
        \node (5) at (1.75, 0.875) {$v_5$};
        \node (6) at (2.875, 0) {$v_6$};
        \node (7) at (4, 0.875) {$v_7$};
        \node (8) at (3.5, 2) {$v_8$};
        \node (9) at (4.625, 2.875) {$v_9$};
        \node (10) at (5.75, 2) {$v_{10}$};
        \node (11) at (5.25, 0.875) {$v_{11}$};
        \node (12) at (6.375, 0) {$v_{12}$};
        \node (13) at (7.5, 0.875) {$v_{13}$};
        \node (14) at (7, 2) {$v_{14}$};
        \node (15) at (8.125, 2.875) {$v_{15}$};
        \node (16) at (9.25, 2) {$v_{16}$};
        \node (17) at (8.75, 0.875) {$v_{17}$};
        \begin{scope}[>={},every node/.style={fill=white,circle,inner sep=0pt,minimum size=12pt}]
        \path [] (1) edge (2);
        \path [] (2) edge node {1} (3);
        \path [] (3) edge node {6} (4);
        \path [] (4) edge (5);
        \path [] (1) edge (5);
        \path [] (5) edge node {2} (6);
        \path [] (6) edge node {7} (7);
        \path [] (7) edge (8);
        \path [] (8) edge (4);
        \path [] (8) edge node {3} (9);
        \path [] (9) edge node {8} (10);
        \path [] (10) edge (11);
        \path [] (11) edge (7);
        \path [] (11) edge node {4} (12);
        \path [] (12) edge node {9} (13);
        \path [] (13) edge (14);
        \path [] (14) edge (10);
        \path [] (14) edge node {5} (15);
        \path [] (15) edge node {10} (16);
        \path [] (16) edge (17);
        \path [] (17) edge (13);
        \end{scope}
        \end{scope}
        
        \end{tikzpicture}
        \caption{Algorithm on a 1D Pentagonal graph of length 5}
%        \label{}
\end{figure}
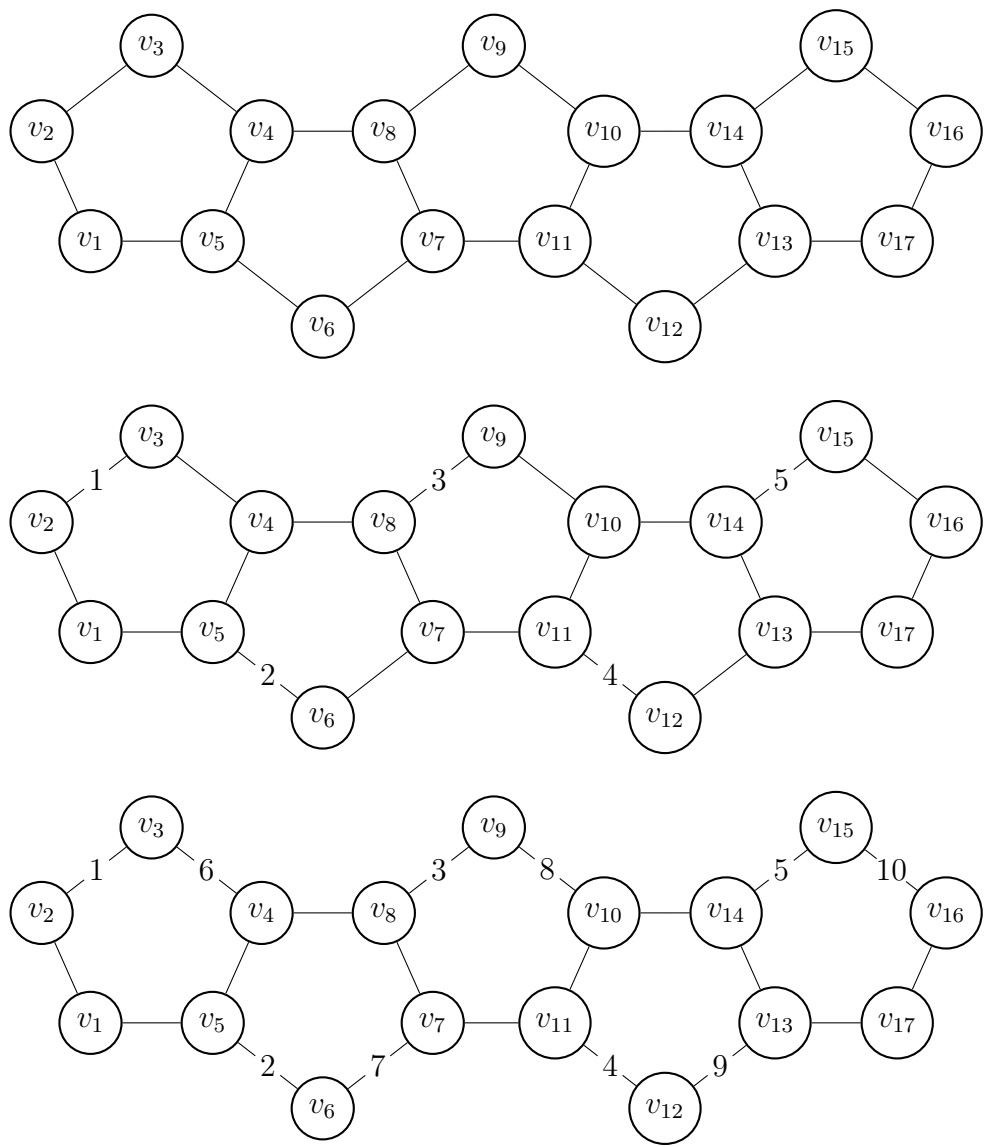
\begin{figure}[htbp]
    \centering
        \begin{tikzpicture}[every node/.style={circle,thick,draw},scale=1.3] 
        
        \begin{scope}[]
        \node (1) at (0.5, 0.875) {$v_1$};
        \node (2) at (0, 2) {$v_2$};
        \node (3) at (1.125, 2.875) {$v_3$};
        \node (4) at (2.25, 2) {$v_4$};
        \node (5) at (1.75, 0.875) {$v_5$};
        \node (6) at (2.875, 0) {$v_6$};
        \node (7) at (4, 0.875) {$v_7$};
        \node (8) at (3.5, 2) {$v_8$};
        \node (9) at (4.625, 2.875) {$v_9$};
        \node (10) at (5.75, 2) {$v_{10}$};
        \node (11) at (5.25, 0.875) {$v_{11}$};
        \node (12) at (6.375, 0) {$v_{12}$};
        \node (13) at (7.5, 0.875) {$v_{13}$};
        \node (14) at (7, 2) {$v_{14}$};
        \node (15) at (8.125, 2.875) {$v_{15}$};
        \node (16) at (9.25, 2) {$v_{16}$};
        \node (17) at (8.75, 0.875) {$v_{17}$};
        \begin{scope}[>={},every node/.style={fill=white,circle,inner sep=0pt,minimum size=12pt}]
        \path [] (1) edge (2);
        \path [] (2) edge node {1} (3);
        \path [] (3) edge node {6} (4);
        \path [] (4) edge (5);
        \path [] (1) edge node {11} (5);
        \path [] (5) edge node {2} (6);
        \path [] (6) edge node {7} (7);
        \path [] (7) edge (8);
        \path [] (8) edge node {12} (4);
        \path [] (8) edge node {3} (9);
        \path [] (9) edge node {8} (10);
        \path [] (10) edge (11);
        \path [] (11) edge node {13} (7);
        \path [] (11) edge node {4} (12);
        \path [] (12) edge node {9} (13);
        \path [] (13) edge (14);
        \path [] (14) edge node {14} (10);
        \path [] (14) edge node {5} (15);
        \path [] (15) edge node {10} (16);
        \path [] (16) edge (17);
        \path [] (17) edge node {15} (13);
        \end{scope}
        \end{scope}
        
        \begin{scope}[yshift=-4cm]
        \node (1) at (0.5, 0.875) {$v_1$};
        \node (2) at (0, 2) {$v_2$};
        \node (3) at (1.125, 2.875) {$v_3$};
        \node (4) at (2.25, 2) {$v_4$};
        \node (5) at (1.75, 0.875) {$v_5$};
        \node (6) at (2.875, 0) {$v_6$};
        \node (7) at (4, 0.875) {$v_7$};
        \node (8) at (3.5, 2) {$v_8$};
        \node (9) at (4.625, 2.875) {$v_9$};
        \node (10) at (5.75, 2) {$v_{10}$};
        \node (11) at (5.25, 0.875) {$v_{11}$};
        \node (12) at (6.375, 0) {$v_{12}$};
        \node (13) at (7.5, 0.875) {$v_{13}$};
        \node (14) at (7, 2) {$v_{14}$};
        \node (15) at (8.125, 2.875) {$v_{15}$};
        \node (16) at (9.25, 2) {$v_{16}$};
        \node (17) at (8.75, 0.875) {$v_{17}$};
        \begin{scope}[>={},every node/.style={fill=white,circle,inner sep=0pt,minimum size=12pt}]
        \path [] (1) edge node {16} (2);
        \path [] (2) edge node {1} (3);
        \path [] (3) edge node {6} (4);
        \path [] (4) edge node {17} (5);
        \path [] (1) edge node {11} (5);
        \path [] (5) edge node {2} (6);
        \path [] (6) edge node {7} (7);
        \path [] (7) edge node {18} (8);
        \path [] (8) edge node {12} (4);
        \path [] (8) edge node {3} (9);
        \path [] (9) edge node {8} (10);
        \path [] (10) edge node {19} (11);
        \path [] (11) edge node {13} (7);
        \path [] (11) edge node {4} (12);
        \path [] (12) edge node {9} (13);
        \path [] (13) edge node {20} (14);
        \path [] (14) edge node {14} (10);
        \path [] (14) edge node {5} (15);
        \path [] (15) edge node {10} (16);
        \path [] (16) edge node {21} (17);
        \path [] (17) edge node {15} (13);
        \end{scope}
        \end{scope}
        
        \end{tikzpicture}
        \caption{Algorithm on a 1D Pentagonal graph of length 5 continued}
%        \label{}
\end{figure}
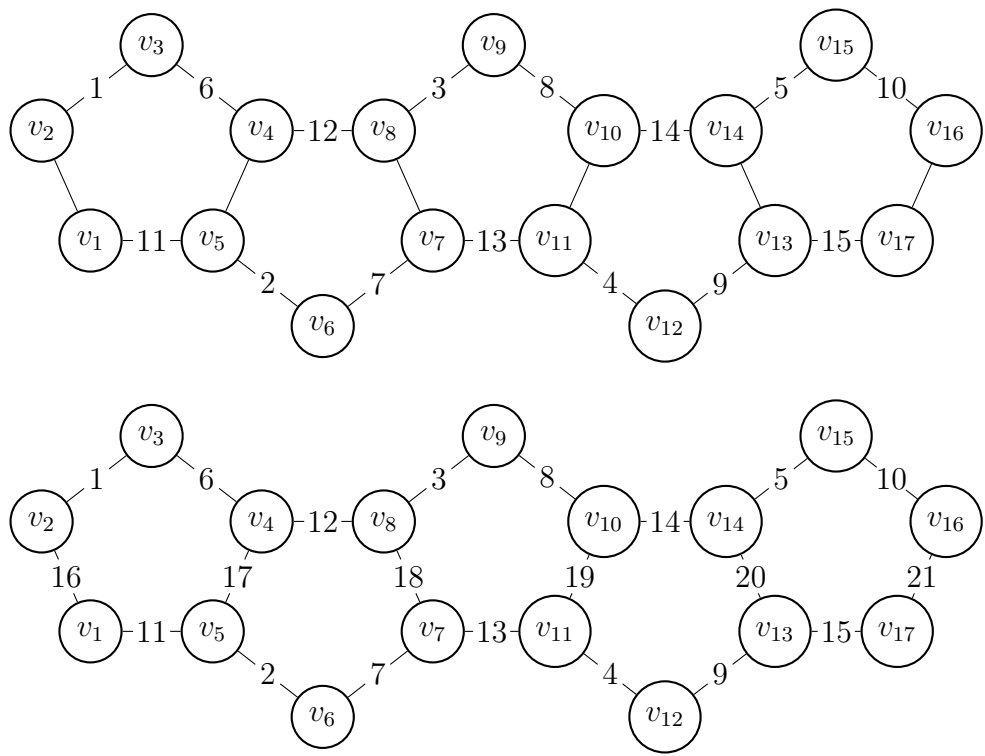

\end{document}